\documentclass{article}

\usepackage{array,diagbox}

\usepackage{pinlabel}
\usepackage{graphicx}
\usepackage{amsmath}
\usepackage{amssymb}
\usepackage{amsthm}
\usepackage{color,colortbl}
\usepackage{hyperref}
\usepackage[english]{babel}
\usepackage{epstopdf}   
\usepackage{cancel}
\usepackage[dvipsnames]{xcolor}

\usepackage{array,multirow, multicol, diagbox}

\usepackage[all,cmtip]{xy}
\usepackage{tikz-cd} 

\definecolor{lightcopper}{rgb}{.93, .76, .58}

\input epsf.tex
\newcommand{\Rone}{\mathbb{Z}}
\newcommand{\Rmor}[1]{\mathbb{Z}^{#1}}
\newcommand{\Tone}[1]{\mathbb{Z}_{#1}}

\newtheorem{theorem}{Theorem}
\newtheorem{proposition}[theorem]{Proposition}
\newtheorem{corollary}[theorem]{Corollary}

\newtheorem{definition}[theorem]{Definition}
\newtheorem{example}{Example}
\newtheorem{remark}[theorem]{Remark}
\newtheorem{notation}[theorem]{Notation}

\renewenvironment{proof}[1][Proof]{\textit{#1.} }
{\hfill \rule{0.5em}{0.5em}}

\begin{document}

\title{New torsion patterns in Khovanov homology}

\author{R. D\'{\i}az and P. M. G. Manch\'on
\footnote{R. Díaz is partially supported by Spanish Research Project PID2020-114750GB-C32/AEI/10.13039/501100011033. P. M. G. Manchón is partially supported by Spanish Research Project PID2020-117971GB-C21.}}
\maketitle


\begin{abstract}
In \cite{RPF} we found some patterns in link diagrams that give rise to torsion elements of order two in their Khovanov homology. In this paper we extend these results by providing new torsion patterns. Many of the torsion elements found in this way have the same homological and quantum degrees; we identify a type of submodules of the Khovanov chain complex that allows us to prove that most of these torsion elements living in the same Khovanov module are really different.

We use the results of this paper together with those in \cite{RPF} to find all the torsion elements in many small twists knots. In addition, we apply them to determine torsion elements in some families of pretzel links, closures of braids with three strands and rational links.
\end{abstract}


\textbf{MSC Class:} \emph{57K10, 57K18.}

\section{Introduction}
This paper is a natural continuation of \cite{RPF}, where we found a pattern of torsion in Khovanov homology. Here we extend these results and introduce the concept of even module, which will be useful to prove that certain elements are not exact and, in addition, to distinguish from
each other the torsion elements found in both papers. The results shown here are, in a sense, complementary to those in~\cite{RPF}.  

Understanding how torsion appears and what geometric meaning it has in the Khovanov homology of links and knots is a relevant problem. In 2004 Shumakovitch \cite{conjectureShumakovitch} conjectured that all links (except the trivial knot, the Hopf link and their disjoint unions and connected sums) have torsion. In \cite{Asaeda} Asaeda and Przytycki proved that certain semi-adequate links have torsion of order two if, roughly speaking, the $A$-smoothing of the semi-adequate diagram has a cycle of order odd or even, the torsion appearing then in the penultimate or antepenultimate quantum degree respectively. They also proved that alternating links other than the Shumakovitch exceptions have torsion of order two. Many other papers have dealt with the problem of torsion in the Khovanov homology since then (for example \cite{Helme-Guizon}, \cite{PrzytyckiSazdanovic}, \cite{LowranceSazdanovic}, \cite{Mukherjee}, \cite{ShumakovitchThinLinksHaveTorsionOnlyOfOrderTwo}, \cite{Chandler}). 

In the present paper we detect torsion elements of order two when a link has a diagram $D$ and a certain smoothing $sD$ of $D$ contains what we will call (blue) ladders (Theorem~\ref{TheoremGreat}). Basically, a ladder appears when $A$-smoothing two twisted strands with positive crossings. Some extra conditions must be satisfied, one of them involving what we will call periphery number of a ladder, a number that tells us whether or not the two extreme arcs that appear when breaking the ladder, belong to the same circle. The proof of Theorem~\ref{TheoremGreat} will be divided into two parts: first we will construct families of chains $X$ such that $dX=2V$ for a certain chain $V$. Then, by using the concept of even module, we will check that~$V$ is not an exact element. Corollary~\ref{CorollaryOnlyOneLadder} guarantees the presence of torsion elements under more relaxed, easier-to-verify assumptions.

Theorem~\ref{TheoremGreat} will provide, in general, a large amount of torsion elements of order two, many of them with the same homological and quantum degrees, hence elements of the same Khovanov homology module. So, a natural question arises: are these torsion elements the same one or are they different? In Theorem~\ref{TheoremDistinguir} we will use again even modules, now to prove that, in general, the torsion elements obtained are different, except for some very particular cases.  

Theorem 3.2 in \cite{RPF} provided another collection of torsion elements for what we called monocircular diagrams of type $D(g,h)$, with a certain extra condition. Again, we will prove that these torsion elements are different from each other, except for some very particular cases. Finally, even modules will be used to show that both collections of torsion elements, those in \cite{RPF} and these ones in the present paper, are disjoint. Both results are collected in Propositions~\ref{PropositionComparacionAnteriores} and \ref{PropositionComparacionAnterioresNuevos}.

Our approach has the benefit of constructing explicitly the torsion elements, as in \cite{RPF}. In \cite{Kindred}, the problem of finding explicit chains that define non-zero elements in the Khovanov homology is addressed. The elements found in \cite{Kindred} are some sort of traces defined as an alternating sum of enhancements of a unique Kauffman state. The torsion elements found in this paper are linear combinations of different enhanced states. Moreover, by contrast to what happens in \cite{Asaeda}, our torsion elements can be found in any homological degree, and they are not restricted to semi-adequate diagrams. 

We will describe many examples of application of our results. These will cover families of pretzel links and rational links, and many closures of braids with three strands. For example, we will see that a rational link $D(a_1, \dots, a_m)$ with each entry $a_i$ even and one of them greater than two, has at least $2^{-m}\prod_{i=1}^m a_i-1$ different torsion elements (although really larger lower bounds can be given). Similar lower bounds are also obtained for the other mentioned families of links, exemplifying the abundance of torsion. In \cite{Sevilla}, for any link closure of a positive braid with three strands, the authors compute the homology modules corresponding to the first four columns and the last three rows in the usual Khovanov table. Moreover, the exact part of the Khovanov table that they can determine depends on the infimum of the braid, a parameter appearing in its left Garside normal form (see \cite{Sevilla} for details). However, a top right part of the table remains to be unknown. We will see that, in many examples, our results determine the presence of torsion in some cells of this unknown part of the Khovanov table. 

The paper is organized as follows: in Section~\ref{SectionKhovanov} we briefly review the combinatorial definition of Khovanov homology due to Viro, which we use later. Readers familiar with the notation used in \cite{RPF}, particularly the homological and quantum degrees $i$ and $j$, can skip this section. In Section~\ref{SectionLadders} we introduce the concepts of (blue) ladder and periphery number of a ladder and we find some chains whose differentials are twice another ones (Proposition \ref{PropositiondX=2V}). In Section~\ref{SectionEvenModules} we introduce the even modules. These modules will be used in Section~\ref{SectionMain} to prove our main result, Theorem~\ref{TheoremGreat}, about detecting explicit torsion elements. In Section~\ref{SectionDistinguishing} we address the problem of distinguishing the torsion elements found (Theorem~\ref{TheoremDistinguir}). In Section~\ref{SectionMonocircularRevisited} we revisit monocircular diagrams; Proposition~\ref{PropositionComparacionAnterioresNuevos} shows that the collection of torsion elements obtained in the present paper is disjoint from that obtained in \cite{RPF}. The last section is dedicated to show different examples.

\section{Khovanov homology} \label{SectionKhovanov}
The Khovanov homology of knots and links was introduced by Mikhail Khovanov at the end of last century (\cite{Khovanov}, \cite{Bar}). In~\cite{Viro} Viro interpreted it in terms of enhanced states of diagrams. We will use Viro's point of view, with some simplifications of the homological and quantum (or polynomial) indexes taken from~\cite{DasbachLowrance}. 

Let $D$ be an oriented diagram of an oriented link $L$, with $p$ positive crossings (\includegraphics[scale=0.1]{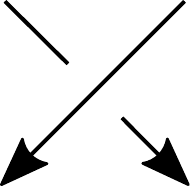}) and $n$ negative crossings (\includegraphics[scale=0.1]{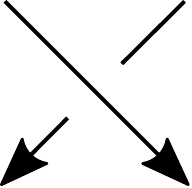}). Let $w(D)=p-n$ be the writhe of $D$.  Let $\text{Cro}(D)$ be the set of crossings of $D$. A (Kauffman) state $s$ of $D$ is a map $s:\text{Cro}(D)\rightarrow \{A,B\}$. If we smooth each crossing of $D$ according to the label of $s$ in the crossing, as shown in Figure~\ref{FigureCrossingSmoothing}, we obtain a set $sD$ with $|sD|$ disjoint simple curves, called the circles of $sD$ (see Figure~\ref{FigureEjemploEstadoRealzado}). We will draw a small chord or {\it scar} (blue if the label is $A$, red if $B$) to remember which was the state. An enhanced state is a pair $(s,e)$ where $s$ is a state and $e$ is an assignment of signs, $-$ or $+$, to each circle of $sD$. For short, we usually write just $s$ instead of $(s,e)$ to refer to a particular enhanced state. Let $\theta(s)$ be the number of circles with sign $+$ minus the number of circles with sign $-$. We then define the homological degree $i(s)$ of an enhanced state $s$ as the number of $B$-labels of the underlying Kauffman state $s$, and its quantum (or polynomial) degree as $j(s)=i(s)+\theta(s)$. 

\begin{figure}[ht!]
	\labellist
	\pinlabel {$A$} at 140 60
	\pinlabel {$B$} at 310 60
	\endlabellist
	\begin{center}
		\includegraphics[scale=0.4]{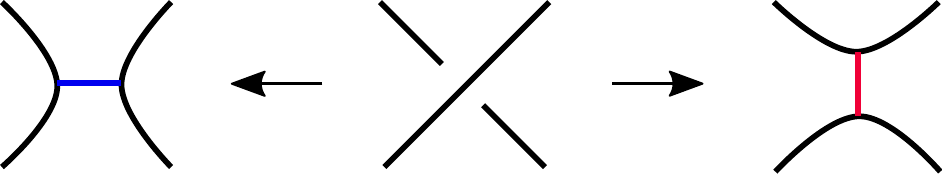}
	\end{center}
	\caption{Smoothing of a crossing according to the label.}\label{FigureCrossingSmoothing}
\end{figure}

\begin{figure}[ht!]
	\labellist
	\pinlabel {$A$} at 50 135
	\pinlabel {$A$} at 193 130
	\pinlabel {$B$} at 125 -8
	\pinlabel {$-$} at 470 140
	\pinlabel {$-$} at 390 50
	\pinlabel {$sD$} at 350 150
	\endlabellist
	\begin{center}
		\includegraphics[scale=0.3]{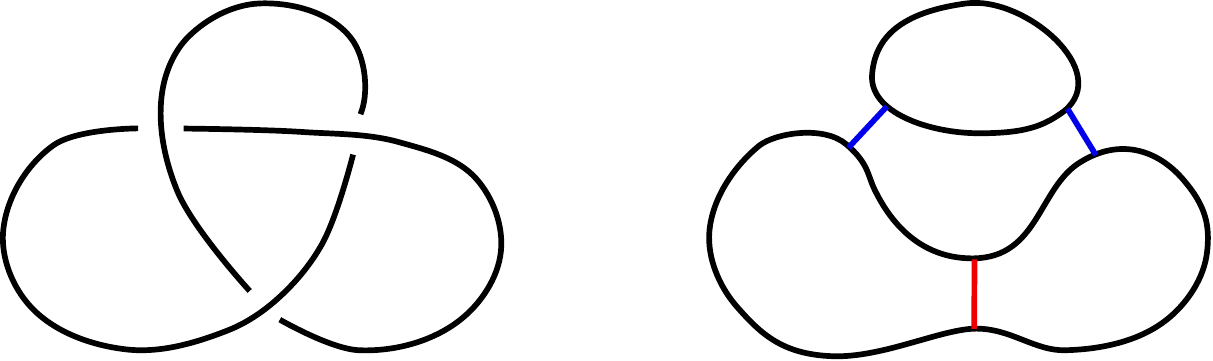}
	\end{center}
	\caption{Enhanced state $s$ with $i(s)=1$, $\theta(s)=-2$ and $j(s)=-1$, hence $s\in C^{1,-1}(D)$. Here $|sD|=2$.} \label{FigureEjemploEstadoRealzado}
\end{figure}

Let $s$ be an enhanced state with $i(s)=i$ and $j(s)=j$. An enhanced state $t$ is adjacent to $s$ if $i(t)=i(s)+1$ and $j(t)=j(s)$, the state $t$ assigns the same labels as $s$ except in one crossing $x=x(s,t)$, where $s(x)=A$ and $t(x)=B$, and $t$ assigns the same signs as $s$ to their common circles. The crossing $x(s,t)$ will be called the change crossing from $s$ to $t$. Passing then from $sD$ to $tD$ can be realized by merging two circles into one, or splitting one circle into two. Affected circles 
are those touching the crossing $x(s,t)$. The possibilities for the signs of these circles, according to the previous definition, are shown in Figure~\ref{FigurePossibilities}.
\begin{figure}[ht!]
	\begin{center}
		\includegraphics[scale=0.5]{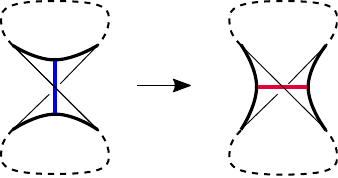} 
		\qquad \qquad
		\includegraphics[scale=0.5]{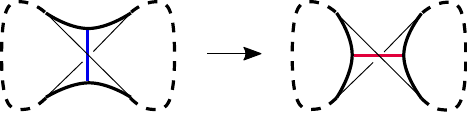}
	\end{center}
	\caption{Possible mergings: $(+,+)\mapsto +$, $(+,-)\mapsto -$ and $(-,+)\mapsto -$. Possible splittings: $+\mapsto (+,-)$ or $+\mapsto (-,+)$ and  $-\mapsto (-,-)$.}
	\label{FigurePossibilities}
\end{figure}

Let $s$ and $t$ be two enhanced states. The incident number $i(s,t)$ of $s$ over~$t$ is defined as follows: if $t$ is adjacent to $s$, then $i(s,t)=(-1)^k$ where $k$ is the number of crossings of~$D$ where $s$ has a $B$-label, previous to the change crossing~$x(s,t)$; otherwise, $i(s,t)=0$. 

Let $R$ be a commutative ring with unit. Let $C^{i,j}(D)$ be the free module over~$R$ generated by the set of enhanced states~$s$ of~$D$ with $i(s)=i$ and $j(s)=j$. Numerate from $1$ to $n$ the crossings of $D$. Now fix an integer $j$ and consider the chain complex
$$
\ldots \longrightarrow C^{i,j}(D) \stackrel{d_i}{\longrightarrow} C^{i+1,j}(D) \longrightarrow \ldots 
$$
with differential $d_i(s)=\sum i(s,t)t$, where the sum runs over all the enhanced states $t$. It turns out that $d_{i+1}\circ d_i=0$. The corresponding homology modules over $R$,
$$
\underline{Kh}^{i,j}(D)=\frac{\text{ker} (d_i)}{\text{im}(d_{i-1})},
$$
are called the Khovanov homology of the diagram $D$ for degrees homological~ $i$ and polynomial~$j$. It turns out that the $R$-modules
$Kh^{h,q}(L) := \underline{Kh}^{i,j}(D)$ where $i=h+n$ and $j=q-p+2n$, are independent of the diagram~$D$ and the order of its crossings; they are the Khovanov homology modules of the oriented link $L$ (\cite{Khovanov}, \cite{Bar}) as presented by Viro \cite{Viro} in terms of enhanced states, and with degrees considered as in \cite{DasbachLowrance}.

One last remark is in order. If $s$ in an enhanced state, then $d(s)=\displaystyle \sum_{x\in \text{Cro}(D)} d_x(s)$ where:
\begin{itemize}
	\item  $d_x(s)=0$ if $s(x)=B$ or $s(x)=A$ and the corresponding $A$-chord in $sD$ joins two different circles, both with sign $-$,
	\item $d_x(s)=(-1)^ks_{x\to B}^{+-} + (-1)^k s_{x\to B}^{-+}$ in case of splitting of a circle $+$,
	\item $d_x(s)=(-1)^ks_{x\to B}^{--}$ in case of splitting of a circle $-$,
	\item $d_x(s)=(-1)^ks_{x\to B}^{+}$ in case of merging of two circles $+$, and
	\item $d_x(s)=(-1)^ks_{x\to B}^{-}$ in case of merging two circles $+-$ or $-+$.
\end{itemize}  
Here $s_{x\to B}$ is the Kauffman state obtained from $s$ by relabeling $s_{x\to B}(x)=~B$, and $k$ is the number of crossings $y\in \text{Cro}(D)$ previous to $x$ such that $s(y)=B$. The signs in the exponent of $s_{x\to B}$ refer to the signs of the circles touching the crossing $x=x(s,s_{x\to B})$.
Finally, we can consider $d_x(Z)$ for a chain $Z$, extending the definition by linearity.

\section{Ladders, periphery numbers and differentiation}  \label{SectionLadders}

Let $s$ be a Kauffman (unenhanced) state of a diagram $D$. Recall that $sD$ is made of some disjoint simple circles and a set of scars, blue or red (corresponding to $A$ or $B$ labels for each crossing of $D$, respectively). A (blue) {\it ladder} $H$ of $sD$ of height $h\geq 1$ is a maximal disc of $\mathbb{R}^2$ that contains exactly two arcs of the same circle or of different circles, and $h$ parallel blue scars (called {\it steps}) joining the two arcs (see Figure~\ref{FigurePeriphery}, left). Here maximal means that there is no larger disc containing an extra parallel blue scar. Cutting all the ladders and joining their extremes as in Figure~\ref{FigurePeriphery}, right, we obtain, for each ladder $H_i$, one or two circles touching it. This circle or these two circles are called the {\it periphery} of $H_i$, and the corresponding number of circles (one or two) is said to be the {\it periphery number} of $H_i$ (see Figure~\ref{FigurePeripheryNumbers} for an example where $s=s_A$, the all $A$-labels state).

\begin{figure}[ht!]
	\begin{center}
		\includegraphics[scale=0.3]{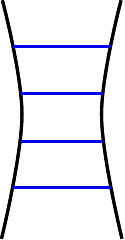}
		\hspace{2cm}
		\includegraphics[scale=0.3]{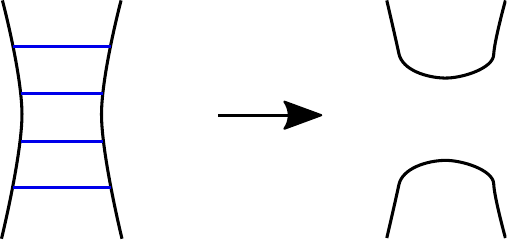}
	\end{center}
	\caption{A (blue) ladder and how to cut it to construct its periphery.}
	\label{FigurePeriphery}
\end{figure}

\begin{figure}[ht!]
	\labellist
	\pinlabel {$D$} at -5 10
	\pinlabel {$s_AD$} at 195 10
	\endlabellist	
	\begin{center}
		\includegraphics[scale=0.45]{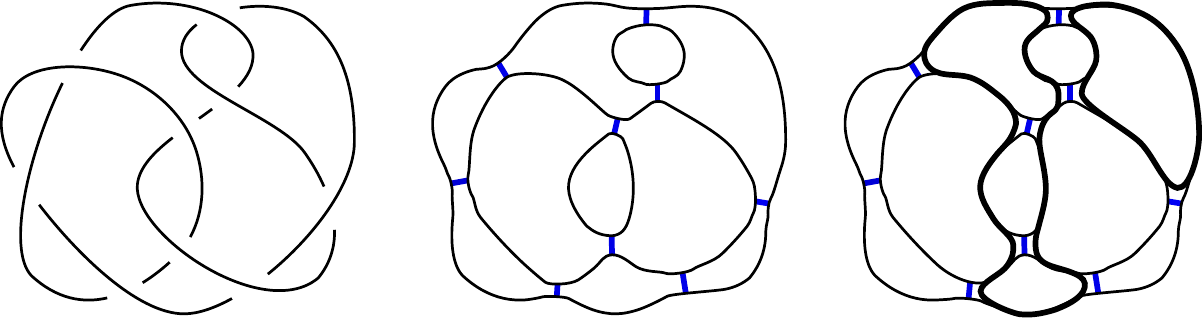}
	\end{center}
	\caption{A diagram $D$ of the knot $9_{42}$ and the corresponding $s_AD$. In $s_AD$ the ladder with three steps has periphery number one, the ladder with two steps has periphery number two; the remaining ladders have one step, and half of them have periphery number one.}
	\label{FigurePeripheryNumbers}
\end{figure}

\begin{remark} \label{RemarkPeripheryNumber}
	We remark that, if $H$ is a common ladder for $s_1D$ and $s_2D$ where $s_1$ and $s_2$ are two states of a diagram $D$, then the periphery numbers of $H$ in $s_1D$ and in $s_2D$ are equal. In fact, the periphery of $H$ remains to be the same if we substitute an $A$-label for a $B$-label or vice versa in any crossing which is not associated to the blue scars of the ladder.
\end{remark}

We now introduce some notation. Let $D$ be a link diagram and fix an initial Kauffman state $s_0$. Suppose that $H_1,\dots, H_k$ is the set of blue ladders of $s_0D$. We order the crossings of $D$ according to the order of the ladders: first the crossings corresponding to $H_1$, then to $H_2$ and so on, and finally the remaining crossings. Moreover, for each ladder the order of the crossings corresponds to the order of the steps in the ladder.  

\begin{notation}
	Let $\mu_1, \dots, \mu_k$ be integers with $0<\mu_i\leq h_i$ for each $i=1,\dots, k$. We denote by $s(\mu_1,\dots,\mu_k;\,)$ the sum of all the possible Kauffman states $s(M_{1}, \dots, M_{k};\,)$ obtained from $s_0$ by choosing $\mu_i$ steps of $H_i$, those in $M_{i}$, and changing them to red (i.e., $B$-label). Note that $s(\mu_1,\dots,\mu_k;\,)$ has $\prod_{i=1}^k\binom {h_i}{\mu_i}$ summands.
\end{notation}

Corresponding to a ladder $H_i$, any state $s= s(M_{1}, \dots, M_{k};\,)$ determines $\mu_i-1$ {\it intermediate circles} $C_i^1, \dots, C_i^{\mu_i-1}$ in $sD$, where the circle $C_i^j$ appears between the $j$-th and the $(j+1)$-th (red)  scars of $M_i$ (see Figure \ref{FigureLadderSmoothing}). In addition, the periphery circle of $H_i$ touching the first red   scar in $M_i$  is denoted by $C_i^0$ and, if the periphery number of $H_i$ is two, the other periphery circle is denoted by $C_i^{\mu_i}$.

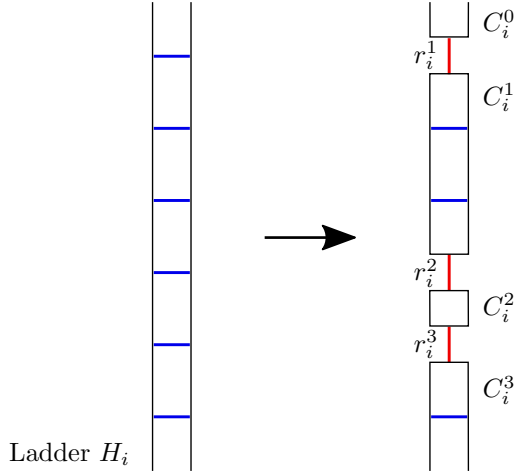
\begin{figure}[ht!]
	\labellist
	\pinlabel {Ladder $H_i$} at -45 10	
	\pinlabel {$C_i^0$} at 190 245
	\pinlabel {$C_i^1$} at 190 205
	\pinlabel {$C_i^2$} at 190 90
	\pinlabel {$C_i^3$} at 190 45
	\pinlabel {$r_i^1$} at 150 228
	\pinlabel {$r_i^2$} at 150 109
	\pinlabel {$r_i^3$} at 150 70	
	\endlabellist	
	\begin{center}
		\includegraphics[scale=0.7]{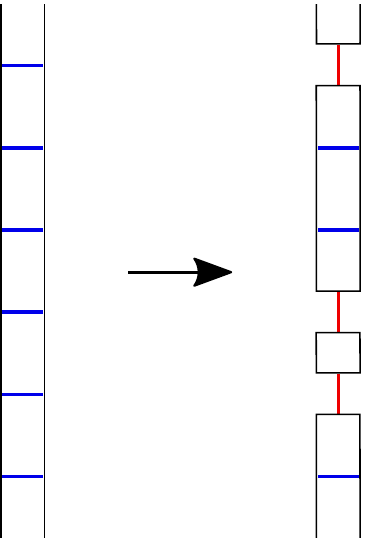}
	\end{center}
	\caption{A ladder $H_i$ with $h_i=6$ steps, and the smoothing  corresponding to $\mu_i=3$ and $M_i=\{1,4,5\}\subset \{1,2,3,4,5,6\}$. The circle $C_i^j$ rests between the $j$-th and the $(j+1)$-th red scars. Note that $C_i^{\mu_i}= C_i^{3}=C_i^0$ if the periphery number of $H_i$ is equal to one.}
	\label{FigureLadderSmoothing}
\end{figure}

\begin{notation}
	Let $D$ be a link diagram. Fix an initial Kauffman state $s_0$. Suppose that $H_1, \dots, H_k$ is the set of blue ladders of $s_0D$. Given integers $\mu_1, \dots, \mu_k$ with $0<\mu_i \leq h_i$ we denote by $s(\mu_1, \dots, \mu_k;+)$ the result of enhancing each circle of each summand of $s(\mu_1,\dots, \mu_k;\,)$ with a sign $+$. We denote by $s(\mu_1,\dots, \mu_k;C_i^{j-})$ the result of enhancing each circle of each summand of $s(\mu_1,\dots, \mu_k;\,)$ with a sign $+$, except the circle $C_i^j$ which is labelled with a sign~$-$ (see Figure~\ref{FigureSuavizados}).
\end{notation} 

\begin{figure}[ht!]
	\labellist
	\pinlabel {$D$} at -15 10
	\pinlabel {$s_0D$} at 90 10
	\pinlabel {\tiny $-$} at 218 28
	\pinlabel {\tiny $-$} at 277 28
	\pinlabel {\tiny $-$} at 337 28
	\pinlabel {\tiny $-$} at 395 28
	\pinlabel {\tiny $-$} at 455 28
	\pinlabel {\tiny $-$} at 514 28
	\endlabellist	
	\begin{center}
		\includegraphics[scale=0.59]{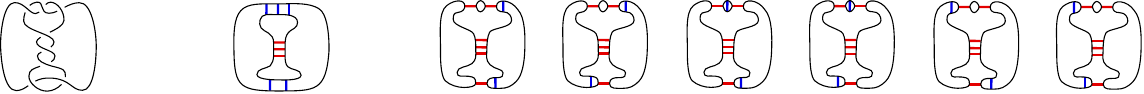}
	\end{center}
	\caption{A diagram $D$, a smoothing $s_0D$ with top ladder~$H_1$ and bottom ladder~$H_2$, and the enhanced states whose sum is the chain $s(2,1;C_1^{0-})$. 
	For example, the fifth enhanced state is $s(\{2,3\},\{1\};C_1^{0-}) = s(\{2,3\},\{1\};C_2^{0-})$.}
	\label{FigureSuavizados}
\end{figure}
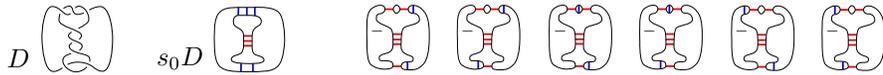

\begin{proposition}\label{PropositiondX=2V}
	Let $D$ be a link diagram. Let $s_0$ be a Kauffman state such that all the blue scars are grouped into $k$ ladders of heights $h_1, \dots, h_k \geq 1$. Assume that the periphery number of each ladder is one. Let $\mu_1, \dots, \mu_k$ be integers with $0<\mu_i\leq h_i$ for all $i=1, \dots, k$. Then $d(X) = 2V$ where
	$$
	X=s(\mu_1, \dots, \mu_k;+)
	$$
	and
	$$
	V= \sum_{\scriptsize{ \begin{matrix} i\in \{1, \dots, k\,/\, \\ \mu_{i} \textnormal{ is even and } \mu_i<h_i\} \end{matrix}}}
	(-1)^{\mu_1+\dots +\mu_{i-1}} s(\dots, \mu_{i-1}, 1+\mu_i, \mu_{i+1}, \dots ;C_i^{0-}).
	$$
	Moreover, $V\not=0$ if and only if there exists $r\in \{1, \dots, k\}$ such that $\mu_r$ is even and $\mu_r<h_r$. 
\end{proposition}  

\begin{proof}
	When differentiating $X$ we change blue chords into red. The hypothesis that $\mu_i>0$ guarantees that, 
	at any summand $S$ of $X$, every blue chord (associated to a crossing $x$) in a ladder (more precisely, a blue scar that was a blue step in a ladder in $s_0D$) is a monochord, and its supporting circle has label $+$ just because all the circles in any summand of $X$ have label $+$. Then, differentiating~$S$ with respect to a step $x$ of a ladder  corresponds to splitting a circle into two new
	ones, one labelled $+$ and the other $-$, so that $d_x(S)$ consists of two states differing only on the signs of these two circles.
	
	Thus, $dX$ is a linear combination of states of the form
	\begin{equation}\label{EquationState}
		s(M_{1},\dots, M_{k}; C_i^{j-})
	\end{equation}
	for some $i=1,\dots, k$ and (since the periphery number is one) some $0\leq j\leq \mu_{i}$, where $M_\alpha\subset \{1, \dots, h_\alpha\}$ has cardinal $1+\mu_\alpha$ if $\alpha =i$, and $\mu_\alpha$ otherwise. 
	
	Now, given an enhanced state $S$ of the form (\ref{EquationState}), we will compute its coefficient in $dX$. 
	For a fixed $i$ we have two cases depending on whether or not $j=0$:
	
	\begin{itemize}
		\item[(i)] Suppose that $j\not=0$. In this case, the circle $C_i^j$ is adjacent to the red chords $r_i^j$ and $r_i^{j+1}$. The enhanced state $S$ appears exactly twice when differentiating $X$: with sign $(-1)^{\mu_1+\dots+\mu_{i-1}+(j-1)}$ when differentiating $s(\dots, M_{i-1}, M_i\setminus\{r_i^j\}, M_{i+1}, \dots ; +)$,
		and with sign $(-1)^{\mu_1+\dots+\mu_{i-1}+j}$ when differentiating $s(\dots, M_{i-1}, M_i\setminus\{r_i^{j+1}\}, M_{i+1}, \dots ; +)$.
		Hence the two occurrences of $S$ cancel each other and $S$ does not appear in $dX$.
			
		\item[(ii)] Assume now that $j=0$. Since the periphery number of $H_i$ is one, the big circle $C_i^0$ is adjacent to both the red (vertical) scars $r_i^1$ and $r_i^{\mu_i+1}$. Then, as before, the state $S$ appears exactly twice when differentiating $X$: with sign $(-1)^{\mu_1+\dots+\mu_{i-1}}$ when differentiating $s(\dots, M_{i-1}, M_i\setminus\{r_i^1\}, M_{i+1}, \dots;+)$, and with sign $(-1)^{\mu_1+\dots+\mu_{i-1}+\mu_i}$ when differentiating $s(\dots, M_{i-1}, M_i\setminus\{r_i^{1+\mu_i}\}, M_{i+1}, \dots;+)$.
		It follows that, if $\mu_i$ is odd, again the two occurrences of $S$ cancel each other, but if $\mu_i$ is even (with $\mu_i<h_i$) then $S$ appears with coefficient $2(-1)^{\mu_1+\dots+\mu_{i-1}}$. This completes the proof. 
	\end{itemize}
\end{proof}

\section{Even modules} \label{SectionEvenModules}
	Let ${\cal M}$ be an $R$-submodule of $C(D)$ generated by a set $B$ of enhanced states. Consider the projection map $\pi_{\cal M}:C(D)\rightarrow C(D)$, defined as the unique $R$-linear map such that, for any enhanced state $Y$, $\pi_{\cal M}(Y)=Y$ if $Y\in B$ and is zero otherwise. The augmentation map $\varepsilon:C(D)\to R$ is the $R$-linear map that sends each enhanced state to 1. As in~\cite{Kindred}, maps of the form $\varepsilon \circ \pi_{\cal M}\circ d:C(D)\rightarrow R$ will be useful to prove that some enhanced states are not exact. 

\begin{definition} 
	Let $D$ be a link diagram. We say that a submodule ${\cal M}$ of $C(D)$ generated by a set $B$ of enhanced states is even if $\varepsilon(\pi_{\cal M}(d(Y)))$ is even for any chain $Y\in C(D)$.
\end{definition} 

We recall that $dY=\sum_{x\in \text{Cro}(D)} d_xY$ where $\text{Cro}(D)$ is the set of crossings of $D$.

\begin{proposition} \label{PropositionEvenExactness}
	Let ${\cal M}$ be a submodule of $C(D)$ generated by a set $B$ of enhanced states. Let $V\in C(D)$ be a chain such that exactly one summand of $V$ is in $B$. Then, if ${\cal M}$ is even, $V$ is not exact.
\end{proposition}
\begin{proof}
	By hypothesis, $\varepsilon\pi_{\cal M}(V)=1$. On the other hand, since  ${\cal M}$ is even, $\varepsilon\pi_{\mathcal M}d(Y)$ is even for any $Y\in C(D)$. It follows that $V\not= d(Y)$ for any $Y\in C(D)$, so $V$ is not exact. 
\end{proof}

\begin{proposition}\label{PropositionConstructingEvenModules}(Constructing even modules).
	Let $B_0$ be a set of (unenhanced) Kauffman states in $C^{i,j}(D)$ such that, for any state $s\in B_0$, all the red scars in $sD$ are bichords. Let $B$ be the set of enhanced states with underlined Kauffman states in $B_0$ and with just one circle labelled $-$. Then the module $\mathcal M$ generated by $B$ is even. 
\end{proposition}
\begin{proof}
	It is enough to show that $\epsilon \pi_{\mathcal M}(dY)$ is even for any enhanced state $Y$, so let $Y$ be an enhanced state in $C^{i-1,j}(D)$ such that $\pi_{\mathcal M}d(Y)\not = 0$. Then there exists a blue scar $b$ in $Y$ such that $\pi_{\mathcal M}d_b(Y)\not= 0$. Then a summand $S$ of $d_b(Y)$ is in $B$. Since all the red scars in the states of $B$ are bichords, $b$ must be a monochord, and since there is only one circle labelled $-$ in the enhanced states in $B$, the  circle $C$ supporting the monochord $b$ must be labelled $+$. Recall that, when differentiating with respect to $b$, the circle $C$ splits in two circles $C_1$ and $C_2$, and $d_bY=S+S'$ where $S$ and $S'$ are exactly equal except that the labels of the circles $C_1, C_2$ are $+,-$ in $S$ and $-,+$ in $S'$. In particular, $S'$ is also in $B$. Hence, we have proved that $\pi_{\mathcal M}d_b(Y)\not=0$ implies that  $\epsilon\pi_{\mathcal M}d_b(Y) = \pm 2$. It follows that $\epsilon\pi_{\mathcal M}d (Y)$ is even.
\end{proof}

\section{Patterns for torsion}  
\label{SectionMain}
\begin{theorem} \label{TheoremGreat}
	Let $D$ be a link diagram of an oriented link $L$. Let $s_0$ be an initial Kauffman state such that all the blue scars in $s_0D$ are grouped into $k$ ladders $H_i$ of heights $h_i\geq 2$, for $i=1,\dots, k$. Suppose that:
	\begin{enumerate}
		\item All the ladders have periphery number one.
		\item There exists a ladder $H_{l}$ with height $h_{l}\geq 3$.
		\item All the red scars of $s_0D$ become bichords when broken all the ladders; precisely, all the red scars of $s_0D$ become bichords in $s_1D$ where $s_1$ is obtained from $s_0$ by changing into red one blue step of each ladder $H_i$, $i=1, \dots, k$. 
	\end{enumerate}
	Then there is an order two torsion element in the Khovanov homology of $L$. Precisely, if $\mu_1, \dots, \mu_k$ are integers with $2\leq \mu_i\leq h_i$, $i=1, \dots, k$, and there exists $r\in \{1, \dots, k\}$ such that $\mu_r$ is even and $\mu_r<h_r$, then the class $[V]$ of the  chain  $V=V(\mu_1,\dots,\mu_k)$ defined in Proposition~\ref{PropositiondX=2V}, that is, 
	$$
	V = \sum_{\scriptsize{ \begin{matrix} i\in \{1, \dots, k\,/\, \\ \mu_{i} \textnormal{ is even and } \mu_i<h_i\} \end{matrix}}}
		(-1)^{\mu_1+\dots +\mu_{i-1}} s(\dots, \mu_{i-1}, 1+\mu_i, \mu_{i+1}, \dots ;C_i^{0-}),
		$$
		is an order two torsion element in $\underline{Kh}^{i,j}(D)$, where $i = i_0 + 1 + \sum_{i=1}^k \mu_i$ and $j = i_0 + |s_1D| + 2\sum_{i=1}^k\mu_i - k$, being $i_0$ the number of red scars in $s_0D$.
\end{theorem}

\begin{proof}
	Let $\mu_1, \dots, \mu_k$ be integers with $2\leq \mu_i\leq h_i$ for all $i=1, \dots, k$, and assume that there exists $r\in \{1, \dots, k\}$ such that $\mu_r$ is even and $\mu_r<h_r$ (note that such a set of integers exist by item 2 and since $h_i\geq 2$). Let $X=s(\mu_1, \dots, \mu_k;+)$. Then, since all the ladders have periphery number one (by item 1) and each $\mu_i\geq 1$, we have that $dX = 2V$ and $V\not=0$ by Proposition~\ref{PropositiondX=2V}. Note that, in particular, $V$ is a cycle.
	
	We now prove that $V$ is not exact. For it, we choose the following enhanced state
	\begin{equation}\label{eq:StateGeneratingSM}
		s=s(\dots, \{1, \dots, \mu_{r-1}\}, \{1, \dots, \mu_r, \mu_r+1\}, \{1, \dots, \mu_{r+1}\}, \dots;C_r^0-),
	\end{equation}
	which is a summand of $s(\dots, \mu_{r-1}, 1+\mu_r, \mu_{r+1}, \dots ;C_r^{0-})$, hence of $V$ since $\mu_r$ is even and $\mu_r<h_r$. Let $B_0=\{S\}$, where $S$ is the underlined Kauffman state of $s$. The red scars of $sD$ are all bichords: if a red scar comes from a blue step of a ladder, it is a bichord since we are assuming  $\mu_i\geq 2$, which is possible since $h_i\geq 2$; if a red scar does not come from a blue step of a ladder, then it is a bichord by item 3. 
	
	Consider now the set $B$ of all enhanced states of $S$ having a unique circle $-$. Then, by Proposition~\ref{PropositionConstructingEvenModules} the module ${\cal M}$ generated by $B$ is an even module. It follows, by Proposition~\ref{PropositionEvenExactness}, that the chain $V$ is not exact.
	
	Finally, note that $i=i_0 + 1 + \sum_{i=1}^k \mu_i$ where $i_0$ is the number of red scars in $s_0D$, the number of circles in each enhanced state of $V$ is equal to $|s_1D|+\sum_{i=1}^k(\mu_i - 1) +1$, hence $\tau = |s_1D|+\sum_{i=1}^k(\mu_i - 1) -1$, and then $j = i +\tau = i_0 + |s_1D| + 2\sum_{i=1}^k\mu_i - k$.
\end{proof}
	
\begin{remark}
	Hypothesis $h_i\geq 2$ in Theorem~\ref{TheoremGreat} is necessary. For example, for $D=P(-1,3)$ and the initial state $s_0=s_A$, we have that $s_0D$ has two ladders, with $h_1=1$ and $h_2=3$, conditions 1 to 3 are satisfied, but the corresponding link, which is the Hopf link, has no torsion at all. See Figure~\ref{FigureHopf}.
\end{remark}
	
\begin{figure}[ht!]
	\labellist
	\pinlabel{$D=P(-1,3)$} at -105 0
	\pinlabel{$s_0D=s_AD$} at 340 0
	\endlabellist
	\begin{center}
		\includegraphics[scale=0.3]{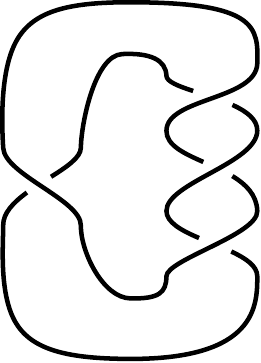}
		\hspace{3cm}
		\includegraphics[scale=0.3]{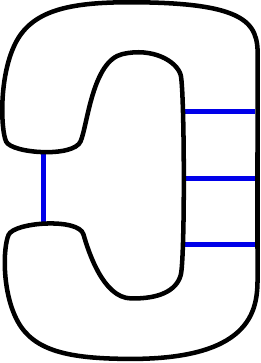}
	\end{center}
	\caption{Diagram $D=P(-1,3)$ of the Hopf link and $s_0D=s_AD$.}
	\label{FigureHopf}
\end{figure}

\begin{corollary}\label{CorollaryOnlyOneLadder}
	Let $D$ be a link diagram of an oriented link $L$. Let $s_0$ be an initial Kauffman state such that all the blue scars in $s_0D$ are grouped into $k$ ladders $H_i$ of heights $h_i\geq 2$, for $i=1,\dots, k$. Suppose that:
	\begin{enumerate}
		\item There exists a ladder $H_{l}$ with periphery number one and height $h_{l}\geq 3$.
		\item All the red scars of $s_0D$ become bichords when broken all the ladders.
	\end{enumerate}
	Then there is an order two torsion element in the Khovanov homology of $L$.
	Precisely, assume that $H_1,\dots, H_t$ are the blue ladders with periphery number~1, $t\leq k$. If $\mu_1, \dots, \mu_t$ are integers with $2\leq \mu_i\leq h_i$, $i=1, \dots, t$, and there exists $r\in \{1, \dots, t\}$ such that $\mu_r$ is even and $\mu_r<h_r$, then the class $[V]$ of the chain $V=V(\mu_1,\dots,\mu_t, h_{t+1}, \dots, h_k)$, defined as in Theorem~\ref{TheoremGreat}, is an order two torsion element.
\end{corollary}

\begin{proof}
	Starting with $s_0$ we consider another state $s'_0$ by changing to red all the steps of all the blue ladders of $s_0D$ with periphery number equal to 2.  We will check that $s'_0D$ satisfies the hypothesis of Theorem \ref{TheoremGreat}. The ladders of $s'_0D$ are the ladders of $s_0D$ with periphery number 1, so there is a ladder in $s'_0D$ with height at least 3. Moreover, the periphery number of these ladders remains to be one by Remark~\ref{RemarkPeripheryNumber}. 
	
	We finally check hypothesis 3 in Theorem~\ref{TheoremGreat}. Let $s'_1$ (resp. $s_1$) be a state obtained from $s'_0$ (resp. $s_0$) by changing into red one blue step of each ladder of $s'_0D$ (resp. $s_0D$). See Figure~\ref{FiguraCorolarioDetalles}. Let $r$ be a red scar in $s'_0D$. If $r$ comes from a blue scar in $s_0D$, then it becomes a bichord in $s'_1D$ since each ladder in $s_0D$ has height greater or equal than two. Otherwise $r$ becomes a  bichord in $s'_1D$ since it is a bichord in $s_1D$ by item {\it 2}, and both $s_1D$ and $s'_1D$ have all the ladders broken. See Figure~\ref{FiguraCorolarioDetalles}.
	
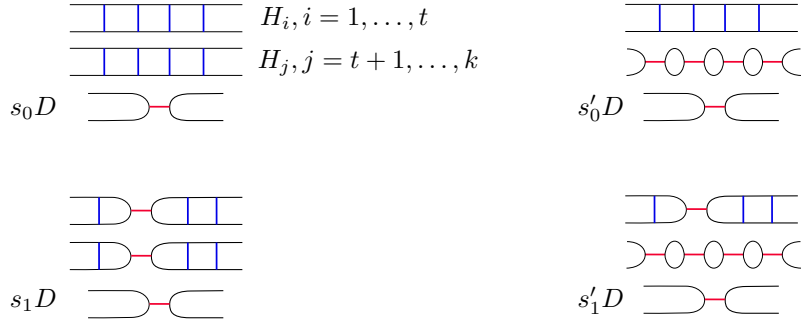
\begin{figure}[ht!]
	\begin{center}
		\labellist
		\pinlabel {$H_i, i=1,\dots, t$} at 150 164
		\pinlabel {$H_j, j=t+1, \dots, k$} at 163 140
		\pinlabel {$s_0D$} at -20 115
		\pinlabel {$s'_0D$} at 290 115
		\pinlabel {$s_1D$} at -20 10
		\pinlabel {$s_1'D$} at 290 10
		\endlabellist
		\includegraphics[scale=0.7]{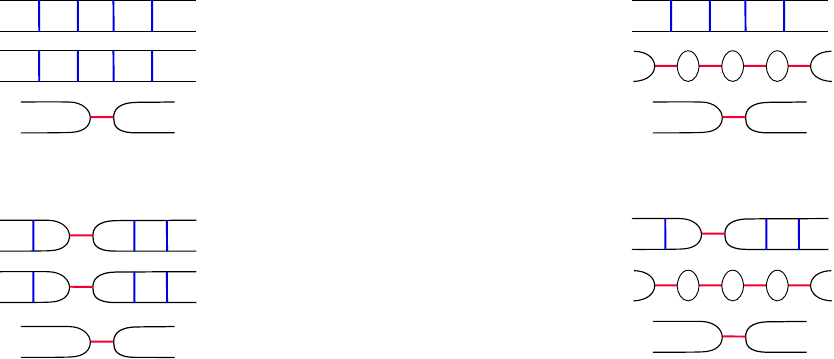}
	\end{center}
	\caption{Typical ladders with periphery numbers one and two and a red scar in $s_0D$, and what they become in $s_1D$, $s'_0D$ and $s'_1D$.}
	\label{FiguraCorolarioDetalles}
\end{figure}	
		
	The last claim is direct from Theorem~\ref{TheoremGreat} applied to the state $s'_0$.
\end{proof}

\begin{example} \label{ExampleSeisUnoMirror}
	In \cite{RPF} we found specific chains representing the torsion elements of the knot $\overline{6_1}$, the mirror image of $6_1$ (see Figure~\ref{FigureKnotMirrorSixOne}, left), for homological degrees $i=2$ and $4$ (see Table~\ref{TableKhovanovHomologyKnotMirrorSixOne}). Now, from Corollary~\ref{CorollaryOnlyOneLadder}, we discover specific chains representing the remaining torsion elements, with degrees $i=5$ and $i=7$.
\begin{figure}[ht!]
	\begin{center}
		\includegraphics[scale=0.35]{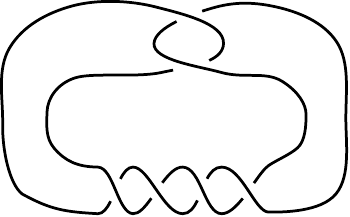}
		\qquad  
		\includegraphics[scale=0.35]{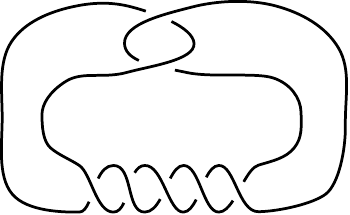}
		\qquad  
		\includegraphics[scale=0.35]{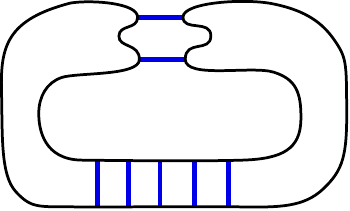}
	\end{center}
	\caption{Knot $\overline{6_1}$, a convenient diagram $D$ for it and the corresponding $s_AD$.}
	\label{FigureKnotMirrorSixOne}
\end{figure}

	\begin{table}[ht!]
	\begin{center}
		\setlength\extrarowheight{2pt}
		\begin{tabular}{|c||c|c|c|c|c|c|c|}
			\hline
			\backslashbox{\!{\color{blue}$j$} \tiny{$q$}\!}{\!{\color{blue}$i$} \tiny{$h$}\!} & {\color{blue}1} \tiny{$-4$} & {\color{blue}2} \tiny{$-3$} & {\color{blue}3} \tiny{$-2$} & {\color{blue}4} \tiny{$-1$} & {\color{blue}5} \tiny{$0$} & {\color{blue}6} \tiny{$1$} & {\color{blue}7} \tiny{$2$} \\
			\hline
			\hline
			{\color{blue}13} \tiny{$5$}&&&&&&&$\Rone$\\
			\hline
			{\color{blue}11} \tiny{$3$} &&&&&&&$\Tone{2}$\\
			\hline
			{\color{blue}9} \tiny{$1$} &&&&&$\Rmor{2}$&$\Rone$&\\
			\hline
			\, {\color{blue}7} \tiny{$-1$}&&&&$\Rone$&$\Rone\oplus\Tone{2}$&&\\
			\hline
			\, {\color{blue}5} \tiny{$-3$}&&&&$\Rone\oplus{\Tone{2}}$ &&&\\
			\hline
			\, {\color{blue}3} \tiny{$-5$}&&$\Rone$&$\Rone$&&&&\\
			\hline
			\, {\color{blue}1} \tiny{$-7$}&&$\Tone{2}$&&&&&\\
			\hline
			\,{\color{blue}-1} \tiny{$-9$}&$\Rone$&&&&&&\\
			\hline
		\end{tabular}
	\end{center}
	\caption{Khovanov homology of the knot $K=\overline{6_1}$.}
	\label{TableKhovanovHomologyKnotMirrorSixOne}
\end{table}
\end{example}

\begin{example} \label{ExampleSieteCuatro}
For a convenient diagram of the knot $7_4$ (notation in \cite{KnotAtlas}) consider the initial state~$s_0$ that assigns two red scars, as shown in Figure~\ref{FigureKnotSieteCuatro}. Theorem~\ref{TheoremGreat} detects then torsion for homological degrees $i=7$ and $9$ (degrees $h=5$ and~$7$ for the knot).
\begin{figure}[ht!]
	\begin{center}
		\includegraphics[scale=0.27]{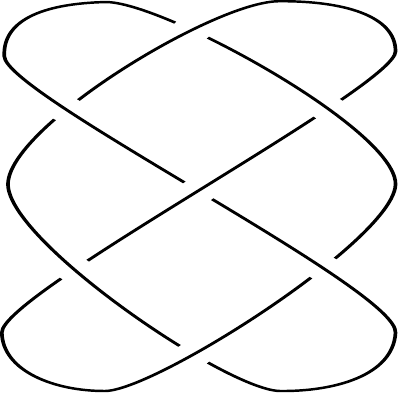}
		\qquad  
		\includegraphics[scale=0.35]{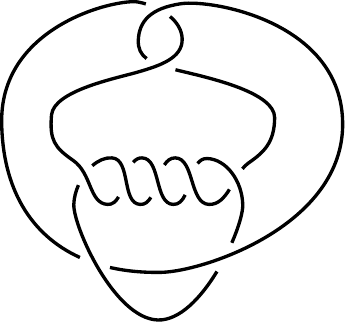}
		\qquad  
		\includegraphics[scale=0.35]{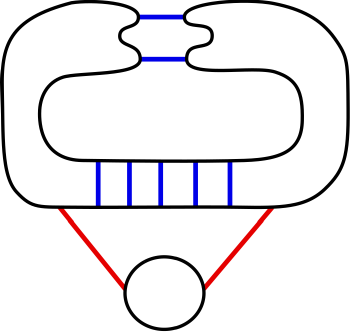}
	\end{center}
	\caption{Knot $7_4$, a convenient diagram $D$ for it and a smoothing $s_0D$.}
	\label{FigureKnotSieteCuatro}
\end{figure}
\end{example}

\begin{remark} \textnormal{Adding circles and red scars in a ``harmless" way, provides new links for which Theorem~\ref{TheoremGreat} guarantees the same subgroups $\mathbb{Z}_2$ of torsion, but for shifted homological and quantum degrees. More precisely, suppose that a diagram link $D$ and an initial Kauffman state $s_0$ are under the hypothesis of Theorem~\ref{TheoremGreat} (or Corollary~\ref{CorollaryOnlyOneLadder}). Now add to $s_0D$ a set of new circles and some red scars; assume that each new red scar has its two endpoints in different circles, and at least one of them is in one of the new circles. Assume, in addition, that no endpoint of the new red scars is between two steps of the same ladder of $s_0D$. The new collection of circles and scars is $s'_0D'$ for a certain diagram $D'$ of a link $L'$. Then, the torsion that Theorem~\ref{TheoremGreat} guarantees for the link $D$ for degrees $i$ and $j$, is also guaranteed for the diagram $D'$ for degrees $i'=i+\alpha$ and $j' = j+\beta$, where $\alpha$ is the number of red scars added, and $\beta$ is $\alpha$ plus the number of circles added. For instance, $s_0D$ in Example~\ref{ExampleSieteCuatro} is obtained from $s_AD$ in Example~\ref{ExampleSeisUnoMirror} by adding in a harmless way one circle and two red scars.} 
\end{remark}	

\section{Even modules can distinguish torsion elements} \label{SectionDistinguishing}

In the previous section we proved that, under the hypothesis of Theorem~\ref{TheoremGreat}, certain elements $V(\mu_1, \dots , \mu_r)$ define torsion elements in the Khovanov homology. Many of these elements are in the same homology module $\underline{Kh}^{i,j}(D)$, so a natural question arises: are the corresponding torsion elements the same element? In Theorem~\ref{TheoremDistinguir} we will use again even modules, now in order to prove that the torsion elements obtained are different, except in some particular cases.  

\begin{theorem} \label{TheoremDistinguir}
	Assume that we are under the hypothesis of Theorem~\ref{TheoremGreat}. Let $(\mu_1, \dots, \mu_k)$ and $(\mu'_1,\dots, \mu'_k)$ be two different uplas satisfying the hypothesis of Theorem~\ref{TheoremGreat}, so that the corresponding chains $V=V(\mu_1,\dots, \mu_k)$ and $V'=V(\mu'_1,\dots, \mu'_k)$ define torsion elements. Then $[V(\mu_1,\dots,\mu_k)] = [V(\mu'_1,\dots,\mu'_k)]$ if and only if the following conditions hold:
	\begin{itemize}
		\item[(i)] $\sum_{t=1}^k\mu_t=\sum_{t=1}^k\mu_t'$
	\end{itemize}
	and there exist $t_1,t_2 \in \{1, \dots, k\}$ such that
	\begin{itemize}
		\item[(ii)] $\mu_{t_1}\not=\mu_{t_1}'$, $\mu_{t_2}\not=\mu_{t_2}'$ and $\mu_t=\mu'_t$ for all $t\not=t_1,t_2$,
		\item[(iii)] for all $t\not= t_1,t_2$, $\mu_t$ is either odd or equal to $h_t$, and
		\item[(iv)] $\mu_{t_1}$ is even, $\mu_{t_1}<h_{t_1}$, $\mu_{t_2}$ is odd and $(\mu'_{t_1},\mu'_{t_2})=(\mu_{t_1}+1,\mu_{t_2}-1)$.
	\end{itemize}
\end{theorem}

\begin{proof}
	Let us first prove the sufficient condition. Items (iii) and (iv) imply that $\mu_{t_1}$ is the only even component of the $k$-upla $(\mu_t)_t$ satisfying $\mu_{t_1}<h_{t_1}$. Thus, according to the definition of $V(\mu_1, \dots, \mu_k)$, we have that
	$$
	V(\mu_1,\dots, \mu_k)=(-1)^{\sigma}s(\dots, \mu_{t_1-1},1+\mu_{t_1}, \mu_{t_1+1},\dots ; C_{t_1}^{0-})
	$$
	for some integer $\sigma$. Now, (ii), (iii) and (iv) imply that $\mu_{t_2}'$ is the only even component of the $k$-upla $(\mu_t')_t$ such that 
	$\mu_{t_2}'<h_{t_2}$. Thus 
	$$
	V(\mu_1',\dots, \mu_k') = (-1)^{\sigma'}s(\dots, \mu_{t_2-1}',1+\mu_{t_2}', \mu_{t_2+1}',\dots ; C_{t_2}^{0-}),
	$$
	for some integer $\sigma'$.  Conditions (ii) and (iv) imply that $V(\mu_1,\dots, \mu_k)$ and $V(\mu_1',\dots, \mu_k')$ are equal, except perhaps for the sign, hence $[V]=\pm[V']$. Since, by Theorem~\ref{TheoremGreat}, both elements have order two, they are the same element. 
	
	We now prove the necessary condition. 
		
	(i) Recall that if $V=V(\mu_1, \dots, \mu_k)$, then $[V] \in \underline{Kh}^{i,j}(D)$ where $i=i_0 + 1 + \sum_{i=1}^k \mu_i$ and $j = i_0 + |s_1D| + 2\sum_{i=1}^k\mu_i - k$, being $i_0$ the number of red scars in $s_0D$. Now observe that $i_0$ and $|s_1D|$ are independent of the upla $(\mu_1,\dots, \mu_k)$, so that $\sum\mu_t = \sum \mu'_t$ if $V$ and $V'$ are homologous.	
		
	The outline of the proof of (ii), (iii) and (iv) is the same. Assuming that  (ii) (respectively (iii), (iv)) is not satisfied, we will prove that $[V]\not= [V']$, by checking that $V-V'$ is not an exact element. To this end, we will use again even modules, in particular Propositions~\ref{PropositionEvenExactness} and \ref{PropositionConstructingEvenModules}. Note that $V$ is a chain in which the red scars of any of its enhanced states are bichords (by the third condition of Theorem~\ref{TheoremGreat} and since $\mu_i\geq 2$). Then, by Proposition~\ref{PropositionConstructingEvenModules}, the module $\mathcal M$ generated by $B$ is even where $B$ is the set of enhanced states obtained by enhancing a single state $s$ of $V$ with just one $-$ circle. Since $V$ has exactly one summand in $B$, to prove that $V-V'$ is not exact it will be sufficient, by Proposition \ref{PropositionEvenExactness}, to show that $V'$ has no summands in $\mathcal M$, that is, $\pi_{\cal M}(V')=0$. For the latter it is sufficient to show that any state $s'$ of $V'$ has, in some ladder, a different number of red scars than the state  $s$ chosen in $V$. That is, if  $\eta_t$ (resp. $\eta'_t$) denotes the number of red scars that $s$ (resp. $s'$) has in the ladder $H_t$, then it is sufficient to show that $(\eta_1,\dots, \eta_k)\not=(\eta_1',\dots, \eta_k')$ for the chosen $s$ and any state $s'$ of $V'$. We remark that, by the definition of $V$ in Theorem~\ref{TheoremGreat}, the $k$-upla $(\eta_t)_t$ is equal to $(\mu_t)_t$ except for one index $r$ for which $\eta_r=\mu_r +1$, being $\mu_r$ even and $\mu_r<h_r$. Similarly for $(\eta_t')_t$.
		
	(ii) Let $s$ and $s'$ be any states in the chains $V$ and $V'$ respectively, and let $(\eta_t)_t$,  $(\eta_t' )_t$ be the $k$-uplas defined as before, so that $(\eta_t)_t$ is equal to $(\mu_t)_t$ except for one component and similarly with $(\eta_t')_t$. Then, if there are three different indices $t_1, t_2, t_3$ such that $\mu_{t_n}\not=\mu'_{t_n}$, $n=1,2,3$, the two uplas $(\eta_1,\dots,\eta_k)$, $(\eta_1',\dots,\eta_k')$ cannot be equal, hence $\pi_{\mathcal M}(V')=0$ and $V-V'$ is not exact, a contradiction. Thus, there are at most two indices $t_1,t_2$ such that $\mu_{t_1}\not=\mu'_{t_1}$ and $\mu_{t_2}\not=\mu'_{t_2}$. Since   $(\mu_1,\dots, \mu_k)\not=(\mu'_1,\dots, \mu'_k)$, at least one component is different; and since $\sum \mu_t=\sum\mu'_t$, there must be a second different component. 	
	
	(iii) Let $t_1,t_2$ be the two indices such that $\mu_{t_i}\not=\mu'_{t_i}, i=1,2$. Suppose that there is a $t\not = t_1, t_2$ such that $\mu_t$ is even and $\mu_t<h_t$. By the definition of $V$, there is a state $s$ in $V$ with $\eta_t = \mu_t + 1$, and therefore $\eta_{t_1} = \mu_{t_1}$ and $\eta_{t_2}=\mu_{t_2}$. Then, for any state $s'$ of $V'$, since there exists $ i \in\{1, 2\}$ such that $\mu'_{t_i}=\eta'_{t_i}$, we find that $\eta'_{t_i}=\mu'_{t_i}\not=\mu_{t_i}=\eta_{t_i}$, so $V-V'$ is not exact, a contradiction.
	
	(iv) By hypothesis of Theorem~\ref{TheoremGreat}, there is some $\mu_t$ even with $\mu_t<h_t$, and condition (iii) implies that $t\in\{t_1,t_2\}$. Assume for instance that $t=t_1$ (hence $\mu_{t_1}$ is even and $\mu_{t_1}<h_{t_1}$). First, we will prove that, if $\mu'_{t_2} \not= \mu_{t_2}-1$, then $V-V'$ is not exact. There is a state $s$ of $V$ with $\eta_{t_1}=\mu_{t_1}+1$. Let $s'$ be an arbitrary state of $V'$. Let $r\in \{ 1, \dots, k\}$ such that $\eta_r'=\mu'_r+1$ (for the state $s'$), so that $\eta'_t=\mu'_t$ for all $t\not=r$. Note that, in particular, $\mu'_r$ is even and $\mu'_r<h_r$, hence $r=t_1$ or $r=t_2$ by (iii). If $r = t_1$, then $\eta_{t_1} = \mu_{t_1}+1 \not= \mu_{t_1}'+1 = \eta_{t_1}'$ by (ii), so that $V-V'$ is not exact. And if $r=t_2$, assuming that $\mu'_{t_2} \not= \mu_{t_2}-1$, we find that $	\eta_{t_2}=\mu_{t_2}\not= \mu'_{t_2}+1=\eta_{t_2}'$, so again $V-V'$ is not exact.
	
	Then, since $\sum \mu_t = \sum \mu'_t$ and $\mu_t=\mu'_t$ if $t\not= t_1, t_2$, we deduce that $\mu_{t_1}' = \mu_{t_1}+1$. Finally, since $\mu'_{t_2}=\mu'_r$ is even and $\mu'_{t_2}=\mu_{t_2}-1$, we find that $\mu_{t_2}$ is odd.
\end{proof}

\begin{remark}\label{RemarkDistinguir}
		If we take $(\mu_1,\dots, \mu_k) $ satisfying the hypothesis of Theorem~\ref{TheoremGreat}, with all the $\mu_i$ even, then Theorem~\ref{TheoremDistinguir}(iv) guarantees that the torsion element $V(\mu_1,\dots,\mu_k)$ is different from all the other torsion elements obtained from Theorem~\ref{TheoremGreat}. We will use this remark to obtain lower bounds on the number of $\mathbb{Z}_2$ subgroups in the Khovanov homology of some families of links. 
\end{remark}

\section{Torsion in monocircular diagrams revisited} \label{SectionMonocircularRevisited}

In Theorem 3.2 of \cite{RPF}, another collection of torsion elements was found for what we called monocircular diagrams of type $D(h_1,h_2)$. In this section we will compare these elements among themselves, and with the new ones introduced in the present paper. 

In \cite{RPF} we called {\it monocircular diagrams} of type $D(h_1,h_2)$ to those oriented link diagrams $D$ whose $s_AD$ has one circle, two blue ladders of heights $h_1,h_2$ and no other ladders nor red scars; hence it is just the standard diagram of the pretzel knot $D(h_1,h_2) = P(-1, \stackrel{(h_1)}{\dots}, -1, h_2)$. For convenience, we state here Theorem 3.2 of \cite{RPF} and the part of Proposition 3.1 involved in it using the notation introduced in this paper.

\begin{theorem} (Theorem 3.2 in \cite{RPF}) \label{Theorem3.2RPF}
	Let $D$ be a monocircular diagram of type $D(h_1,h_2)$ with $h_1,h_2 \geq 2$. Let $\mu$ be an odd integer with $1\leq \mu < h_i$ where $i\in \{1, 2\}$. Then the classes of the elements $V(0,\mu):=s(1,\mu; C_2^{0-})$ if $i=1$, and $V(\mu,0):=(-1)^{\mu}s(\mu,1; C_1^{0-})$ if $i=2$, define torsion elements of order two in $\underline{Kh}^{\mu+1,2\mu-1}(D)$. 
\end{theorem}

Note in particular that $[V(0,\mu)]$ and $[V(\mu, 0)]$ are in the same homology group, and there are no other coincidences. In the following proposition we will show that these two elements are different with one single exception, the case in which $\mu=1$. Subsequently, we will prove that they are always different from the torsion elements obtained in the present paper, in Theorem~\ref{TheoremGreat}. Again, even modules will play a key role.

\begin{proposition}\label{PropositionComparacionAnteriores}
	The elements $V(\mu, 0)$ and $V(0,\mu')$ define the same torsion element if and only if $\mu=\mu'=1$.
\end{proposition}
\begin{proof}
	If $\mu=\mu'=1$, then $V(0,1)=s(1,1;C_1^{0-})=-V(1,0)$ so they define the same torsion element. Conversely, if $[V(\mu,0)]=[V(0,\mu')]$, we already know that $\mu=\mu'$. Suppose that $\mu > 1$. In the proof of Theorem 3.2 of \cite{RPF}, we considered the submodule $\mathcal M$ generated by the subset of chains
	\begin{eqnarray*}
		B&=&\{  s(\{\},\{1,\dots ,\mu+1\}; C^0_2,C^k_2) \,/\, k=1,\dots, \mu+1\}
		\\
		&& \cup \,  \{s(\{1\},\{2,\dots ,\mu+1\}; C^k_2)\,/\, k=0,1,\dots, \mu-1\},
	\end{eqnarray*}
	where we have used the notation introduced in Section~\ref{SectionLadders}; for instance, $s(\{1\},\{2,\dots ,\mu+1\};  C^k_2)$ is the enhanced state whose first step in the ladder $H_1$ and the steps $2$ to $\mu+1$ in the ladder $H_2$ are red, and the circle $C^k$ of the ladder $H_2$ is labelled $-$. We showed there that $\mathcal M$ is even, that is, $\varepsilon(\pi_{\mathcal M}(d(Y)))$ is even for any chain $Y$. In particular, if $V=V(0,\mu)$ and $V'=V(\mu,0)$ were homologous, $\varepsilon(\pi_{\mathcal M}(V-V'))$  should be even. Now, 
	$V(0,\mu)=s(1,\mu;C_2^{0-})$, so that $\pi_{\mathcal M}(V) = s(\{1\},\{2,\dots,\mu+1\};C_2^{0-})$ and $\varepsilon(\pi_{\mathcal M}(V))=1$. 
	
	On the other hand, $V(\mu,0)=(-1)^{\mu}s(\mu,1; C_1^{0-})$ so that all its states have $\mu$ red scars in the ladder $H_1$. Since $\mu>1$, we get $\pi_{\mathcal M}(V')=0$ and $\varepsilon(\pi_{\mathcal M}(V-V'))=1$. Hence $V,V'$ are not homologous.  
\end{proof}  

\begin{proposition}\label{PropositionComparacionAnterioresNuevos}
	Let $(\mu_1,\mu_2)$ be a pair of integers such that $V(\mu_1,\mu_2)$ defines a torsion element for $D(h_1,h_2)$, according to Theorem~\ref{TheoremGreat}. Let $\mu$ be an odd integer such that $\mu<h_1$, so $V(\mu, 0)$ defines a torsion element according to Theorem~\ref{Theorem3.2RPF}. Then $[V(\mu_1,\mu_2)]\not=[V(\mu,0)]$. Analogously $[V(\mu_1,\mu_2)]\not=[V(0,\mu)]$ if $\mu$ is odd and $\mu<h_2$. 
\end{proposition}

\begin{proof}
	Assume that $\mu_1$ is even and $\mu_1<h_1$ (the case $\mu_2$ even with $\mu_2<h_2$ is similar). Let $\mathcal M$ be the even module considered in the proof of Theorem~\ref{TheoremGreat}, that is, $\mathcal M$ is generated by a subset $B$ of different enhancements of the state $s=s(\{1,2,\dots, \mu_1+1\},\{1,2,\dots, \mu_2\};)$. Then $\varepsilon(\pi_{\mathcal M}(V(\mu_1,\mu_2)))=1$. Now, any state in $V(\mu,0)$ has one ladder with exactly one red scar while, since $\mu_1,\mu_2\geq 2$, any state in $s$ has at least two red scars in both ladders. Hence $\varepsilon(\pi_{\mathcal M}(V(\mu,0)))=0$ and $V(\mu,0)$ and $V(\mu_1,\mu_2)$ are not homologous, by the same argument as in Proposition~\ref{PropositionComparacionAnteriores}.
\end{proof}

	We can collect all the above information in a figure so that it will be easy to grasp it. Denote by $G_1$ the set of pairs $(\mu_1,\mu_2)$ satisfying the hypothesis of Theorem~\ref{Theorem3.2RPF}, so $V(\mu_1,\mu_2)$ is a torsion element according to this theorem. Analogously, let $G_2$ be the set of pairs $(\mu_1,\mu_2)$ so $V(\mu_1,\mu_2)$ is a torsion element according to Theorem~\ref{TheoremGreat}. Precisely, 
	$$
	G_1 =  \{(0,\mu) \,/\, 1\leq \mu< h_2, \mu \hbox{ odd} \} \,\cup\, \{(\mu,0) \,/\, 1\leq \mu< h_1,  \mu \hbox{ odd}\}
	$$
	and 
	$$
	G_2 = \{(\mu_1,\mu_2) \in\mathbb{Z}^2 \,/\, 2\leq \mu_i\leq h_i, \hbox{with } \mu_2 \hbox{ even, } \mu_2< h_2 \hbox{ or } \mu_1 \hbox{ even, } \mu_1< h_1 \}.
	$$
	
	Note that, if $(\mu_1,\mu_2)\in G_2$, then the homological and quantum degrees of $V(\mu_1,\mu_2)$ are $i=\mu_1+\mu_2+1$ and $j=2i-3$, according to Theorem~\ref{TheoremGreat}, once observed that $i_0=0$ and $|s_1D|=1$ for monocircular diagrams. If $(\mu,0)$ or $(0,\mu) \in G_1$, we had $i=\mu +1$ and $j=2i-3$ again.
	
	Now we represent the pairs $(\mu_1,\mu_2)$ in a grid $G(h_1,h_2)$, the pairs in $G_1$ by red points, the pairs in $G_2$ by blue points. For example, the grid $G(10,15)$ for the pretzel diagram with 25 crossings $D(10,15)=P(-1, \stackrel{(10)}{\dots}, -1, 15)$ is shown in Figure~\ref{FiguraRejilla}. Points on the same line with slope $-1$ are in the same Khovanov module (the corresponding homological degree appears in grey). We join two of these points through a green segment if and only if they provide the same torsion element, according to Propositions~\ref{PropositionComparacionAnteriores} and \ref{PropositionComparacionAnterioresNuevos}. For $D(10,15)$, the number of $\mathbb{Z}_2$ components in the different degrees from $i=1$ to $i=25$  coincides with the torsion obtained for this link from KnotInfo \cite{KnotInfo}. They are the following:
	\begin{center}
	\begin{tabular}{|c|c|c|c|c|c|c|c|c|c|c|c|c|c|c|}
		\hline 
		$i$ & 1 & 2 & 3 & 4 & 5 & 6 & 7 & 8 & 9 & 10 & 11 & 12 & 13 & 14 \\ 
		\hline 
		& 0 & 1 & 0 & 2 & 1 & 3 & 2 & 4 & 3 & 5 & 4 & 5 & 5 & 5 \\ 
		\hline 
	\end{tabular} 
	\begin{tabular}{|c|c|c|c|c|c|c|c|c|c|c|c| }
		\hline 
		$i$ & 15 & 16 & 17 & 18 & 19 & 20 & 21 & 22 & 23 & 24 & 25 \\ 
		\hline 
		& 5 & 4 & 5 & 4 & 4 & 3 & 3 & 2 & 2 & 1 & 1  \\ 
		\hline 
	\end{tabular} 
	\end{center}

	\begin{figure}[ht!]
		\labellist
		\pinlabel {\scriptsize  $1$} at 29 21
		\pinlabel {\scriptsize $10$} at 165 21
		\pinlabel {\scriptsize  $\mu_1$} at 203 21
		\pinlabel {\scriptsize  $1$} at 6 46 
		\pinlabel {\scriptsize  $15$} at 2 256
		\pinlabel {\scriptsize  $\mu_2$} at 5 285
		\pinlabel {\textcolor{gray}{\scriptsize  $i=1$}} at 36 -5  
		\pinlabel {\textcolor{gray}{\scriptsize  $11$}} at 200 -5
		\pinlabel {\textcolor{gray}{\scriptsize  $16$}} at 202 73
		\pinlabel {\textcolor{gray}{\scriptsize  $25$}} at 202 208
		\endlabellist
		\begin{center}
			\includegraphics[scale=0.6]{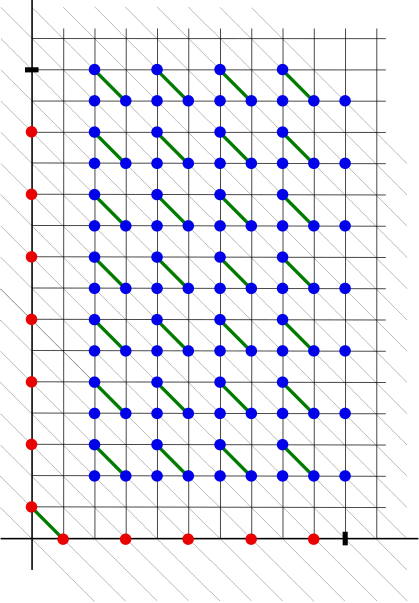}
		\end{center}
		\caption{The grid $ G(10,15)$ for the pretzel knot $D(10, 15)$.}
		\label{FiguraRejilla}
	\end{figure}
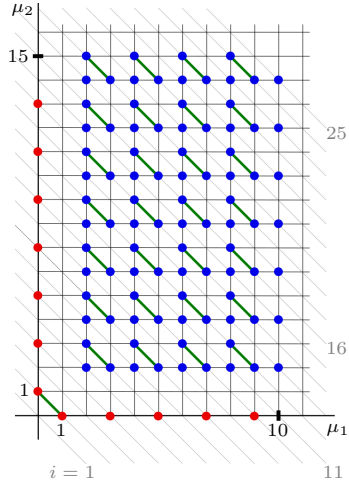
	
	For the general case, one can easily observe a pattern for the grids $G(h_1,h_2)$ of $D(h_1, h_2)=P(-1, \stackrel{(h_1)}{\dots}, -1, h_2)$. One must take care near the border lines $\mu_1=h_1$ and $\mu_2=h_2$, where the parity of $h_1,h_2$ matters. The cases with $h_1=2$ or $h_2=2$ are a bit special. For example, Figure~\ref{FiguraVariasRejillas} shows the grids $G(2,5)$, $G(2,6)$, $G(4,6)$ and $G(5,7)$, that describe the torsion (given by Theorems~\ref {TheoremGreat} and \ref{Theorem3.2RPF}) for the pretzels $D(2,5)$, $D(2,6)$, $D(4,6)$ and $D(5,7)$. We can extract some easy consequences. For instance, if $h_1, h_2>2$, the pretzel link $D(h_1,h_2)$ has torsion in all the columns of its Khovanov table except in the first and third one, and except in the last one if both $h_1$ and $h_2$ are even. Moreover, for a homological index $i$ with $2\leq i \leq \min\{h_1,h_2\}$, the number of $\mathbb Z_2$-components provided by Theorems \ref{Theorem3.2RPF} and \ref{TheoremGreat} in $\underline{Kh}^{i,2i-3}(D)$ is $\frac{i-3}{2}$ if $i$ is odd, $\frac{i}{2}$ if $i$ is even. In particular, there are pretzel links $D(h_1,h_2)$ with as many $\mathbb{Z}_2$-components as desired in the same Khovanov module.
		
	\begin{figure}[ht!]
		\labellist
		\pinlabel {\scriptsize  $2$} at 45 5  
		\pinlabel {\scriptsize  $2$} at 188 5  
		\pinlabel {\scriptsize  $4$} at 365 5  
		\pinlabel {\scriptsize  $5$} at 529 5  
		\pinlabel {\scriptsize  $5$} at 5 92  
		\pinlabel {\scriptsize  $6$} at 148 107 
		\pinlabel {\scriptsize  $6$} at 295 107  
		\pinlabel {\scriptsize  $7$} at 445 122  
		\pinlabel {\textcolor{gray}{\scriptsize  $7$}} at 124 -4 
		\pinlabel {\textcolor{gray}{\scriptsize  $7$}} at 265 -4 
		\pinlabel {\textcolor{gray}{\scriptsize  $10$}} at 416 26  
		\pinlabel {\textcolor{gray}{\scriptsize  $12$}} at  564 71
		\endlabellist
		\begin{center}		
			\includegraphics[scale=0.6]{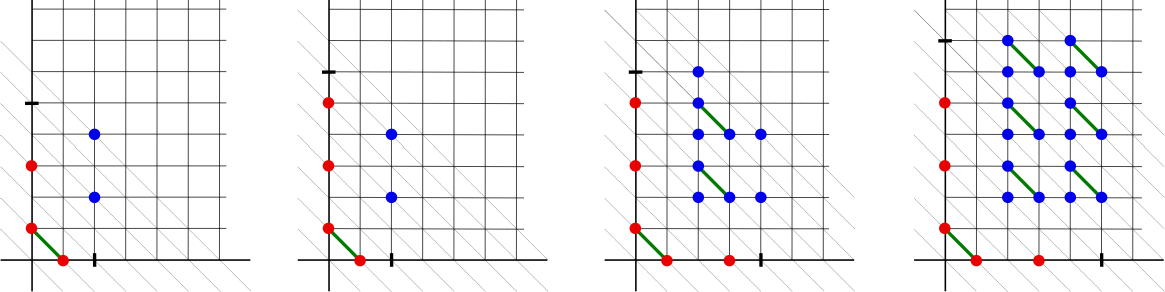}
		\end{center}
		\caption{The grids $G(2,5)$, $G(2,6)$, $G(4,6)$ and $G(5,7)$.}
		\label{FiguraVariasRejillas}
	\end{figure}
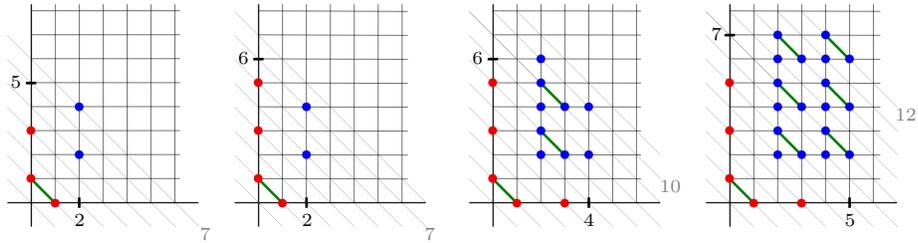	
	
	As it has been said, for $D(10,15)$ our theorems provide all the torsion. This is also true for many small (with few crossings) pretzel diagrams $D(h_1, h_2)$, and it seems an easy exercise to prove that it is true in general, by considering the long exact sequence of Khovanov homology obtained by smoothing any crossing.  

	As a summary example, consider the knot $8_2^*$, the mirror image of $8_2$. This is a pretzel knot with diagram $D=D(3,6)=P(-1,-1,-1,6)$. See Figure~\ref{FigureOchoDosEspecular}. Its smoothing $s_AD$ has two ladders. Since $D$ has $n=3$ negative crossings and $p=6$ positive crossings, we have that $Kh^{h,q}(8_2^*) = \underline{Kh}^{i,j}(D)$ where $i = h+n =h+3$ and $j = q-p+2n =q$. The Khovanov homology of $8_2^*$ is shown in Table~\ref{TableKnotEightTwoMirror}.
	\begin{figure}[ht!]
		\begin{center}
			\includegraphics[scale=0.65]{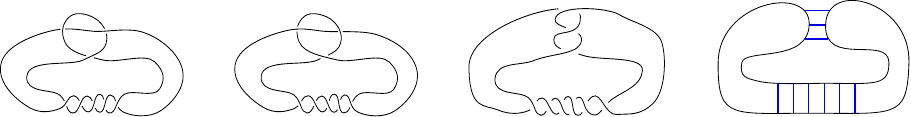}
		\end{center}
		\caption{$8_2$, $8_2^*$, a good diagram $D=P(-1,-1,-1,6)$ of $8_2^*$ and $s_AD$.}
		\label{FigureOchoDosEspecular}
	\end{figure}
		
	\begin{table}[ht!]
		\setlength\extrarowheight{2pt}
		\begin{tabular}{|c||c|c|c|c|c|c|c|c|c|}
			\hline
			\backslashbox{\!{\color{blue}$j$}\,{\tiny$q$}}{\!{\color{blue}$i$}\,{\tiny  $h$}} & {\color{blue}$1$} {\tiny $-2$}&{\color{blue}$2$} {\tiny $-1$} & {\color{blue}$3$} {\tiny $0$} & {\color{blue}$4$} {\tiny $1$} & {\color{blue}$5$} {\tiny $2$} & {\color{blue}$6$} {\tiny $3$} & {\color{blue}$7$} {\tiny $4$} & {\color{blue}$8$} {\tiny $5$}  & {\color{blue}$9$} {\tiny $6$ }\\
			\hline
			\hline
			{\color{blue} $17$} {\tiny 17} &   &   &   &   &   &   &   &   & $ \Rone $ \\
			\hline
			{\color{blue} $15$} {\tiny 15} &   &   &   &   &   &   &   & $ \Rone $ & $ {\color{blue}\Tone{2}} $ \\
			\hline
			{\color{blue} $13$} {\tiny 13} &   &   &   &   &   &   & $ \Rone $ & $ \Rone \oplus {\color{blue}\Tone{2}} $ &   \\
			\hline
			{\color{blue} $11$} {\tiny 11} &   &   &   &   &   & $ \Rmor{2} $ & $\Rone \oplus {\color{blue}\Tone{2}}$ &   &   \\
			\hline
			{\color{blue} $9$} {\tiny 9} &   &   &   &   & $ \Rone $ & $ \Rone \oplus {\color{purple}\Tone{2}} \oplus {\color{blue}\Tone{2}} $ &   &   &   \\
			\hline
			{\color{blue} $7$} {\tiny 7} &   &   &   & $ \Rone $ & $ \Rmor{2} \oplus {\color{blue}\Tone{2}} $ &   &   &   &   \\
			\hline
			{\color{blue} $5$} {\tiny 5} &   &   & $ \Rone $ & $ \Rone \oplus {\color{purple}\Tone{2}} $ &   &   &   &   &   \\
			\hline
			{\color{blue} $3$} {\tiny 3} &   & $ \Rone $ & $ \Rmor{2} $ &   &   &   &   &   &   \\
			\hline
			{\color{blue} $1$} {\tiny 1} &   & $ {\color{purple}\Tone{2}} $ &   &   &   &   &   &   &   \\
			\hline
			{\color{blue}\!\!\!$-1$} {\tiny $-1$} & $ \Rone $ &   &   &   &   &   &   &   &   \\
			\hline
		\end{tabular}
		\caption{Khovanov homology of the knot $8_2^*$.}
		\label{TableKnotEightTwoMirror}
	\end{table}
	
	On the one hand, {\color{purple}$\Tone{2}$} torsions in purple are explained by Theorem~\ref{Theorem3.2RPF}, by just taking $\mu = 1, 3$ and $5$ (the pairs $(0,1)$ or $(1,0)$, $(0,3)$ and $(0,5)$), corresponding to homological degrees $i=\mu+1=2,4$ and $6$. 
	On the other hand, {\color{blue}$\Tone{2}$} torsions in blue are explained by Theorem~\ref{TheoremGreat}. The pairs $(\mu_1, \mu_2) =(2,2)$, $(2,3)$  or $(3,2)$, $(2,4)$, $(2,5)$ or $(3,4)$ and $(2,6)$ provide torsion corresponding to the homological degrees $i=5, 6, 7, 8$ and $9$. Finally, the two torsion elements found for the homological degree $i=6$ are different according to Proposition~\ref{PropositionComparacionAnterioresNuevos}. We can easily draw the corresponding grid $G(3,6)$ and find out that the torsion components obtained from it are exactly those appearing in the table. 
		
\section{Pretzel links, $3$-braids and rational links}

In this section we apply our results to find that, in general, there is an enormous amount of $\mathbb{Z}_2$ torsion subgroups in the Khovanov homology of many pretzel links, closures of $3$-braids and rational links. With respect to pretzel links \cite{Lickorish}, we remark that, by the results in \cite{Asaeda} one can deduce that most of them have torsion (in the penultimate or antepenultimate row of the Khovanov homology table). The following result emphasizes the large amount of torsion subgroups (and recall that they can appear in any row):

\begin{corollary}\label{CorollaryPretzel}
	Let $a_1, \dots, a_n$ be integers different from zero. Let $P(a_1, \dots, a_n)$ be a pretzel diagram of a pretzel link $L$. Suppose that:
	\begin{enumerate}
		\item $a_l\not=1$ for any $l\in \{1, \dots, n\}$.
		\item There exists $i\in \{1, \dots, n\}$ such that $a_i\geq 3$.
		\item There are $j,k\in\{1, \dots, n\}$, $j\not= k$, such that $a_j, a_k<0$.
	\end{enumerate}
	Then there are at least $\displaystyle{\prod_{l/a_l>0}} \lfloor {\frac{a_l}{2}} \rfloor -1$ different torsion subgroups $\mathbb{Z}_2$ in the Khovanov homology of $L$.	
\end{corollary}

\begin{proof}
	Let $s_0$ be the initial Kauffman state that assigns $A$-labels to the positive crossings and $B$-labels to the negative crossings (see Figure~\ref{FigurePretzelInitialState}).
	\begin{figure}[ht!]
		\labellist
		\pinlabel {$a_i>0$} at -70 10
		\pinlabel {$a_i<0$} at 450 10
		\endlabellist	
		\begin{center}
			\includegraphics[scale=0.3]{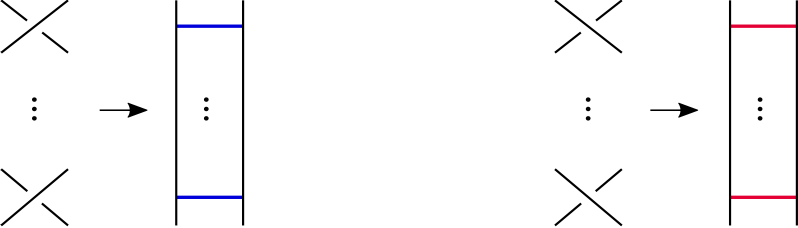}
		\end{center}
		\caption{Initial state for a pretzel diagram.}
		\label{FigurePretzelInitialState}
	\end{figure}
	
	In $s_0D$ we group the blue scars into $r$ ladders, corresponding to the positive entries $a_{l_1}, \dots, a_{l_r}$. All these ladders have periphery number equal to one, since there is at least one negative entry (see Figure~\ref{FigurePretzelPeripheryNumberOne}). There is also a ladder $H$ with height $h\geq 3$ by item {\it 2}. Finally, all the red scars are bichords when all the ladders are broken, by item {\it 3}. 
	\begin{figure}[ht!]
		\begin{center}
			\includegraphics[scale=0.4]{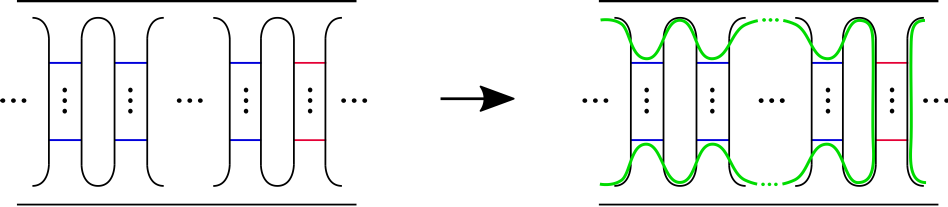}
		\end{center}
		\caption{Ladders have periphery number one.}
		\label{FigurePretzelPeripheryNumberOne}
	\end{figure}
	
	Theorem~\ref{TheoremGreat} guarantees then a torsion element for each upla $(\mu_1, \dots, \mu_r) \in \mathbb{Z}^r$, if $2\leq \mu_s \leq a_{l_s}$, $s=1, \dots, r$, and at least one $\mu_t$ is even and less than $a_{l_t}$. Moreover, by choosing only even numbers $\mu_s$, $s=1, \dots, r$,  all these uplas define different torsion elements, by Remark~\ref{RemarkDistinguir}.
\end{proof}

\begin{remark} \label{RemarkNotSharp}
	The bound given in Corollary~\ref{CorollaryPretzel} is not sharp at all. For example, for $P(5,-3,2,3,-2)$, the provided bound is just one, but Theorem~\ref{TheoremGreat} gives torsion elements for $(\mu_1, \mu_2, \mu_3)$ equals to $(2,2,2)$, $(2,2,3)$, $(4,2,2)$ and $(4,2,3)$, and all these elements are different from each other by Theorem~\ref{TheoremDistinguir}. 
\end{remark}

We now focus on $3$-braids. 

\begin{corollary}\label{CorollaryBraids}
	Let $L=\hat{\beta}$ be the closure of a braid $\beta$ with three strands, defined by the word $w = \sigma_1^{a_1} \sigma_2^{a_2} \dots \sigma_1^{a_{2t-1}} \sigma_2^{a_{2t}}$ where all the exponents are non zero. Assume the following: 
	\begin{enumerate}
		\item No exponent $a_i$ is equal to 1.
		\item There exists an exponent $a_j>2$ cyclically surrounded by positive ones. 
		\item Any negative exponent is cyclically surrounded by positive ones. 
	\end{enumerate}
	Then the Khovanov homology of $L$ has at least $\lfloor \frac{a_{i_1}}{2}\rfloor \dots \lfloor \frac{a_{i_r}}{2}\rfloor-1$ different torsion elements, where $a_{i_1},\dots, a_{i_r}$ are the positive exponents surrounded cyclically by positive exponents.
\end{corollary}

\begin{proof}
	We check the hypothesis of Corollary~\ref{CorollaryOnlyOneLadder}, considering the initial state~$s_0$ that assigns $A$-labels to the positive crossings, $B$-labels to the negative crossings. Indeed, item {\it 1} in this corollary implies that the heights of all the blue ladders are greater than or equal to two; item {\it 2} implies the existence of a ladder with height at least 3 and periphery number 1 (Figure~\ref{FiguraBraids}, left); and item {\it 3} implies that all the red scars become bichords when broken the ladders (Figure~\ref{FiguraBraids},~right).
	
	\begin{figure}[ht!]
		\begin{center}
			\includegraphics[scale=0.45]{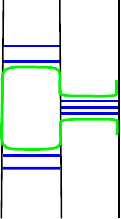}
			\hspace{3cm}	
			\includegraphics[scale=0.45]{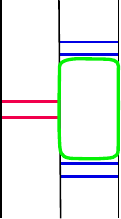}	
		\end{center}
		\caption{Large ladder with periphery number one; red scars are bichords.}
		\label{FiguraBraids}
	\end{figure}
	
	Finally, note that the ladders corresponding to the subscripts $i_1,\dots, i_r$ are precisely those with periphery number one (see again Figure~\ref{FiguraBraids}). Corollary~\ref{CorollaryOnlyOneLadder} and Remark~\ref{RemarkDistinguir} provide then the lower bound given for the number of order two torsion elements. 					
\end{proof}

In \cite{Sevilla}, the authors study the Khovanov homology of the closure of positive braids with three strands. Since conjugate braids close to the same link, they study the Khovanov tables up to conjugacy, finding that a certain block of the table is repeated a number of times that depends on the {\it infimum} $p$ of the braid, a parameter appearing in its Garside left normal form. When written the braid as $\Delta^p a_1\cdots a_l$, with factors $a_1, \dots, a_l$, they smooth the first crossing in $a_1$ via $A$ and $B$ smoothing, obtaining new diagrams $D_A$ and $D_B$, and apply the exact sequence of Khovanov homology in order to prove that there is an isomorphism between the corresponding Khovanov modules of $D$ and $D_A$, since the Khovanov homology of the adjacent modules for $D_B$ turn to be zero (which is proved by transforming $D_B$ in an appropriate alternating standard rational link diagram, while keeping trace of the writhe). Then they conclude by induction on the length of the braids, since $D_A$ is in the same conjugacy class as $D$, and contains one less crossing.

However, in several cases a part of the Khovanov table remains to be unknown. As a consequence of Theorem~\ref{TheoremGreat}, we can prove that there is torsion of order two in many of these unknown cells, provided $p=0$. For example, consider the braid $\beta = \sigma_1^{k_1}\sigma_2^2$ with $k_1 \geq 3$, whose Khovanov table is Table 31 in \cite{Sevilla}. Then, by Corollary~\ref{CorollaryBraids}, there is at least $\lfloor {\frac{k_1}{2}} \rfloor -1$ different torsion subgroups $\mathbb{Z}_2$. Moreover, all of them have homological degrees $h=i\geq 5$ by Theorem~\ref{TheoremGreat}, so they are in the unknown part (note that $i=h$ since $n=0$). For example, if $k_1=7$, we find torsion corresponding to the pairs $(\mu_1, \mu_2)$ equal to $(2,2)$, $(4,2)$ and $(6,2)$ with respective degrees $i=h=5,7$ and $9$. 

Next we concentrate in rational links \cite{Peter}. Since all rational links are alternating, we already know that they have torsion of order two. In the following example we will show that the presence of this torsion is enormous in general.

\begin{example}
	Let $D=D(a_1, \dots, a_m)$ be a standard diagram of a rational link~$L$. Assume that $a_i\geq 2$ for all $i\in \{1, \dots, m\}$ and there exists $j\in \{1, \dots, m\}$ such that $a_j\geq 3$. Then there are at least $\prod_{i=1}^m \left\lfloor \frac{a_i}{2}\right\rfloor -1$ different subgroups $\mathbb{Z}_2$ in the Khovanov homology of $L$. 
\end{example}

In order to see this, consider the initial state $s_0=s_A$. All the blue scars are then grouped into $m$ ladders, and one of them has at least three steps. Moreover, all these ladders have periphery number one (see Figure~\ref{FiguraRationalSADPeriphery}). Then the chain $V=V(\mu_1, \dots, \mu_m)$, where each $\mu_i$ is even, $2\leq \mu_i \leq a_i$ and there exists $r\in \{1, \dots, m\}$ such that $\mu_r<a_r$, determines an order two torsion element. The fact of choosing each $\mu_i$ to be even is not necessary, but it ensures that the homology elements obtained are different from each other, according to Remark~\ref{RemarkDistinguir}. This provides the above lower bound (not sharp at all) for the number of $\mathbb{Z}_2$ subgroups. For example, for $D(4,2,6)$ we can choose $(\mu_1, \mu_2, \mu_3)$ equal to
$(2,2,2)$, $(2,2,4)$, $(2,2,6)$, $(4,2,2)$ and $(4,2,4)$, but also $(2,2,3)$, etc.

\begin{figure}[ht!]
	\begin{center}
		\includegraphics[scale=0.4]{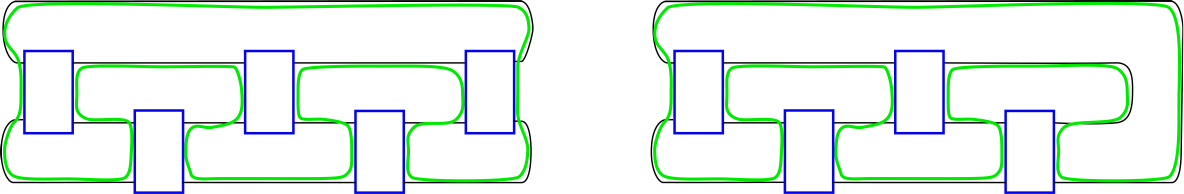}
	\end{center}
	\caption{In both cases, $m$ odd and even, the periphery number of all ladders is one in $s_AD$ if each $a_i>0$.}
	\label{FiguraRationalSADPeriphery}
\end{figure}

Finally, we show a general result for rational links without the restriction of having all the entries positive. The proof by Asaeda and Przytycki~\cite{Asaeda} of the fact that alternating links have torsion of order two examines cycles in $s_AD$, when $D$ is chosen to be an $A$-adequate diagram. The interest of the following result is to discover order two torsion in many rational links starting with a non semi-adequate diagram. 

\begin{proposition}
	Let $D=D(a_1, \dots, a_m)$ be a standard diagram of a rational link, where $a_i \in \mathbb{Z}$, $a_i\not= 0$ for all $i\in \{ 1, \dots, m\}$. Assume the following:
	\begin{enumerate}
		\item There is no $a_i$ equal to one.
		\item There exists a large positive entry surrounded by positive entries. Precisely, there exists $j\in \{ 2, \dots, m-1\}$ such that $a_j\geq 3$ and $a_{j-1}, a_{j+1} >0$, or $a_1 \geq 3$ and $a_2 >0$ or $a_m\geq 3$ and $a_{m-1} >0$.
		\item Any negative entry is surrounded by positive ones. Precisely, if $a_i<0$ with $i\in \{2, \dots, m-1\}$, then $a_{i-1}, a_{i+1} >0$; if $a_1<0$ then $a_2>0$, and if $a_m<0$ then $a_{m-1} >0$.
	\end{enumerate}
	Then the rational link represented by $D$ has torsion of order two in its Khovanov homology.
\end{proposition}
\begin{proof}
	Considering the initial state $s_0$ suggested by Figure~\ref{FiguraCrossingsScars}, the result follows by direct inspection of the hypothesis of Corollary~\ref{CorollaryOnlyOneLadder}. See Figures~\ref{FiguraBecomeBichords} and \ref{FiguraRationalPeripheryNumberOne}.
	
	\begin{figure}[ht!]
		\begin{center}
			\includegraphics[scale=0.3]{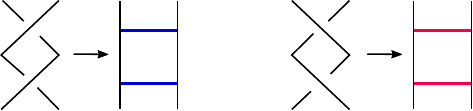}
		\end{center}
		\caption{Crossings and scars.}
		\label{FiguraCrossingsScars}
	\end{figure}
	
	\begin{figure}[ht!]
		\begin{center}
			\includegraphics[scale=0.30]{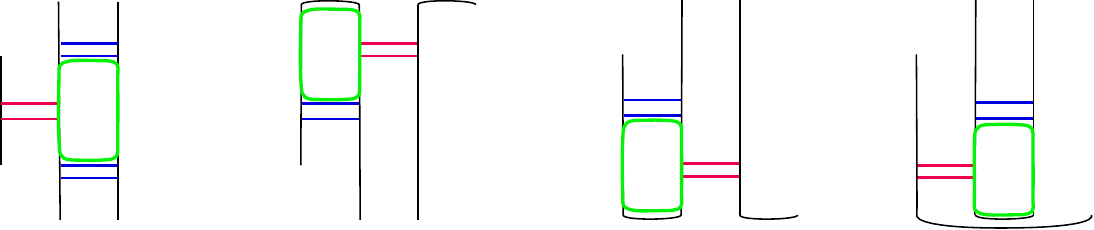}
		\end{center}
		\caption{The red scars become bichords.}
		\label{FiguraBecomeBichords}
	\end{figure}
	
	\begin{figure}[ht!]
		\begin{center}
			\includegraphics[scale=0.30]{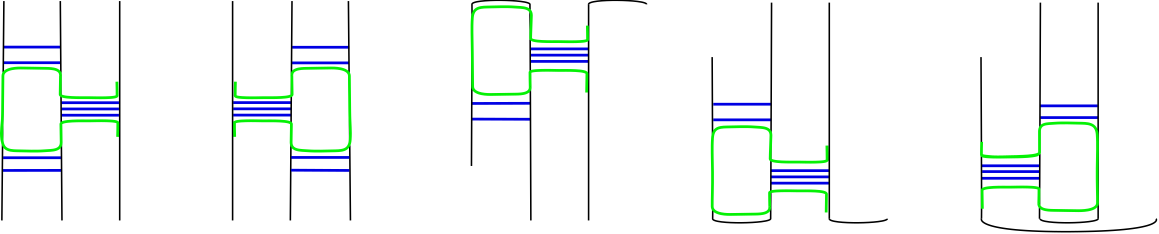}
		\end{center}
		\caption{Periphery Number One.}
		\label{FiguraRationalPeripheryNumberOne}
	\end{figure}
\end{proof}

\end{document}